\documentclass[11pt,reqno,a4paper]{amsart}

\usepackage{color}
\usepackage{ulem}
\usepackage{amssymb,epsfig,graphics}

\newtheorem{theorem}{Theorem}[section]
\newtheorem{proposition}[theorem]{Proposition}

\newtheorem{corollary}[theorem]{Corollary}

\newtheorem{sub-lemma}[theorem]{Sub-Lemma}

\def\L{\mathcal{L}}

\def\1{{{\mathit 1} \!\!\>\!\! I} }

\begin{document}

\title[]{Mixing rate in infinite measure for $\mathbb Z^d$-extension,
      application to the periodic Sinai billiard}
\author{Fran\c{c}oise P\`ene}
\address{1)Universit\'e de Brest, Laboratoire de
Math\'ematiques de Bretagne Atlantique, CNRS UMR 6205, Brest, France\\
2)Fran\c{c}oise P\`ene is supported by the IUF.}
\email{francoise.pene@univ-brest.fr}
%\urladdr{}
\keywords{}
\subjclass[2000]{Primary: 37B20}
\begin{abstract}
We study the rate of mixing of observables of $\mathbb Z^d$-extensions of probability preserving dynamical systems.
We explain how this question is directly linked to the local limit theorem and establish a rate of mixing for general classes of observables of the $\mathbb Z^2$-periodic Sinai billiard. We compare our approach with the induction method.
\end{abstract}
\date{\today}
\maketitle
\bibliographystyle{plain}

A measure preserving dynamical system $(X,f,\mu)$ is given by
a measure space $(X,\mu)$ and a measurable
$\mu$-preserving transformation $f$. 
Given such a dynamical system, the study of its mixing properties means the study of the behaviour of quantities of the following form:
\begin{equation}\label{MIXING}
\int_X u.v\circ f^n\, d\mu\, \quad\mbox{ as }n\rightarrow +\infty.
\end{equation}
When $\mu$ is a probability measure, $(X,f,\mu)$
is said to be mixing if, for every $u,v\in\L^2(\mu)$, 
\eqref{MIXING} converges to the product of integrals $\int_X u\, d\mu\,\int_X v\, d\mu$. 

Assume from now on that $\mu$ is a $\sigma$-finite measure.
As pointed out by
\cite{KS}, there is no reasonable generalization of mixing. Nevertheless, it makes sense to investigate the behaviour of 
\eqref{MIXING}.
More precisely, we are interested in proving that \eqref{MIXING} suitably normalized converges to $\int_X u\, d\mu \int_X v\, d\mu$. We call mixing rate
the corresponding normalization.

Mixing rates (and refined estimates) in infinite measure have been studied by 
Thaler \cite{Thaler}, Melbourne and Terhesiu \cite{MT}, Gou\"ezel
\cite{Gouezel}, Bruin and Terhesiu \cite{BT}, Liverani and Terhesiu
\cite{LT} for a wide family of dynamical systems
including the Liverani-Saussol-Vaienti maps, etc. The method used by these authors is induction.

We emphasize here on the fact that, in the context of $\mathbb Z^d$-extensions, such results are related to precised local limit theorems see
\cite{GH,SV1,SV2}. In particular mixing properties for the periodic planar Sinai billiard have been established
in \cite[Prop. 4]{FP09a} and in \cite[Prop. 4.1]{ps10} for indicator functions of some bounded sets, with three different applications in \cite{FP09a,FP09b,ps10}. We are interested here in stating such results for general functions (with full support). We will present our general approach and use it to establish a mixing rate for a general class of functions in the context of the $\mathbb Z^2$-periodic Sinai billiard.

In some sense, these two approaches are converse one from the other.
Indeed, whereas for the first mentioned method, the mixing rate follows from an estimate of the tail distribution of some return time; for the method we use here, we first prove the mixing rate and
can deduce from it the asymptotic behaviour of the tail distribution of the first return time (see \cite[Thm. 1]{DSV} and \cite[Prop. 4.2]{ps10}).

For both methods, the link between tail distribution return time and 
mixing is given by a renewal equation.
\section{Mixing via induction}
The strategy of the proof via induction consists:
\begin{enumerate}
\item to consider a set $Y\subset X$ of finite measure 
satisfying nice properties; in particular, $\left(\mu\left(Y\cap\{\varphi>n\}\right)\right)_n$  is regularly varying, where $\varphi$ is the first return time to $Y$: $\varphi(y):=\inf\{n\ge 1\ :\ f^ny\in Y\}$.
\item to prove good estimates
for $R_n:v\mapsto \mathbf 1_Y L^n\left(\mathbf 1_{Y\cap \{\varphi=n\}}v \right)$, where $L$ is the transfer operator of $u\mapsto u\circ f$, which is defined by $\int_XLu.v\, d\mu=\int_Xu.\circ f.v\, d\mu$.
\item to deduce from the estimates of $R_n$ and from the renewal equation:
$$\forall n\ge 1,\quad T_n=\sum_{j=1}^n  T_{n-j}R_j,\quad
      \mbox{with}\ \  T_n:v\mapsto \mathbf 1_Y L^n\left(\mathbf 1_Y v\right)$$
an estimate on $T_n$ of the following form:
$$T_n\sim \mu(Y\cap f^{-n}Y)\mathbb E_\mu[\cdot\mathbf 1_Y],$$
on some Banach space $\mathcal B$ of functions $w:X\rightarrow\mathbb C$.
\item to deduce:
$$\int_X v.T_nu\, d\mu=\int_X v.L^nu\, d\mu=\int_X u.v\circ f^n\, d\mu,$$
for every $u,v:X\rightarrow \mathbb C$ supported in $Y$ such that $v\in\mathcal B$
and $w\mapsto \int_Y u.w\, d\mu$ is in $\mathcal B'$.
\item to go to the general situation (functions with full support in $X$) by considering the sets
$A_k:=f^{-k}Y\setminus\bigcup_{\ell=0}^{k-1}f^{-\ell}Y$.
\end{enumerate}

\section{$\mathbb Z^d$-extensions: local limit theorem and mixing}
We consider from now on the special case where $(X,f,\mu)$ is a $\mathbb Z^d$-extension of
a probability preserving dynamical system $(\bar X,\bar f, \bar\mu)$ by $\psi: \bar X\rightarrow \mathbb Z^d$, that is
$X=\bar X\times \mathbb Z^d$, $f(x,k)=(\bar f(x),k+\psi(x))$
and $\mu=\bar\mu\otimes \lambda_d$, where $\lambda_d$
is the counting measure on $\mathbb Z^d$.
Observe that $f^n(x,k)=(\bar f^n(x),k+S_n(x))$, with $S_n:=\sum_{k=0}^{n-1}\psi\circ \bar f^k$.
We set $Y:=X\times \{0\}$.

The crucial idea in this context is to consider a situation where
$(S_n/a_n)_n$ converges in distribution to a stable random variable $B$ and the strategy is then:
\begin{enumerate}
\item to prove a 
local limit theorem (LLT):
$$\forall \ell\in\mathbb Z^d,\quad
     \bar\mu(S_n=\ell)=(\Phi_B(\ell/a_n)+o(1))a_n^{-d}\, ,\mbox{ as }n\rightarrow +\infty\, ,$$
where $\Phi_B$
is the density function of $B$,
and more precisely a "spectral LLT":
\begin{equation}\label{spectralTLL}Q_{n,\ell}:=P^n\left(\mathbf 1_{\{S_n=\ell\}}\cdot\right)= \frac{\Phi_B(\ell/a_n)\mathbb E_{\bar\mu}[\cdot]+\varepsilon_{n,\ell}}{a_n^d}\, ,\quad
\mbox{with}\ \lim_{n\rightarrow +\infty}\sup_\ell \Vert\varepsilon_{n,\ell}\Vert=0\, ,
\end{equation}
on some Banach space $(\mathcal B,\Vert\cdot\Vert)$ of functions $w:\bar X\rightarrow\mathbb C$,
where $P$ is the transfer operator of $\bar f$ (see \cite[Lem. 2.6]{damiensoaz}
for a proof of such a result in a general context).

The following identity makes a relation between 
$Q_{n,0}$ and the operator $T_n$ presented in the previous section:$$(T_{n} v)(x,\ell)=(Q_{n,0}(v(\cdot,0)))(x)\mathbf 1_{\ell=0}.$$

Note that the LLT is already a decorrelation result since:
$$\int_X\mathbf 1_Y.\mathbf 1_Y\circ f^n\, d\mu=\mu\left(Y\cap f^{-n}Y\right)=\bar\mu(S_n=0).$$
\item to use \eqref{spectralTLL} and the definition of $P$ to deduce a mixing result:
$$
\int_X \mathbf 1_{\bar X\times\{k\}}u. (v\mathbf 1_{\bar X\times\{\ell\}})\circ f^n\, d\mu
= \int_{\bar X} u(x,k). v(\bar f^n(x),\ell)\mathbf 1_{\{S_n(x)=\ell-k\}}\, d\bar\mu(x)$$
\begin{eqnarray}
&=& \int_{\bar X} v(\cdot,\ell) P^n\left(u(\cdot,k). \mathbf 1_{\{S_n(\cdot)=\ell-k\}}\right)\, d\bar\mu(x)\nonumber\\
&=&    \Phi_B(0)a_n^{-d}\int_{\bar X\times\{k\}} u\, d\mu\,
         \int_{\bar X\times\{\ell\}} v\, d\mu+o(a_n^{-d})\, ,  \label{DECO1}
\end{eqnarray}
valid for every $k,\ell\in\mathbb Z^d$ and for every $u,v$ such that $u_\ell:=u(\cdot, \ell)\in \mathcal B$
and such that $w\mapsto \int_{\bar X} v_k(y)w(y)\,d\bar\mu(y)$
is in $\mathcal B '$, with $v_k(y):=v(y,k)$, since $\Phi_B$
is continuous.
\item to generalize this as follows:
\begin{eqnarray*}
\int_X u. v\circ f^n\, d\mu &=& \sum_{\ell,m}\int_X \mathbf 1_{\bar X\times\{k\}}u_\ell. (v_m\mathbf 1_{\bar X\times\{m\}})\circ f^n\, d\mu\\
&=&a_n^{-d}\left(o(1)+\sum_{k,m}
    \Phi_B\left(\frac{k-m}{a_n}\right)\int_{\bar X\times \{\ell\}}u\, d\mu \, \int_{\bar X\times \{m\}}v\, d\mu\right) \\
&=&\Phi_B(0)a_n^{-d}\int_X u\, d\mu\,
         \int_X v\, d\mu+o(a_n^{-d})\, ,
\end{eqnarray*}
which holds true as soon as 
$\sum_\ell\Vert u_\ell\Vert_{\mathcal B}<\infty$
and $\sum_\ell\Vert \mathbb E_{\bar\mu}[v_\ell\, \cdot ]\Vert_{\mathcal B'}<\infty$,
since $\Phi_B$ is continuous and bounded, where we used again the notations $u_\ell:=u(\cdot,\ell)$ and $v_m:=v(\cdot,m)$.
\item to go from \eqref{spectralTLL} to the study of 
$\bar\mu(\varphi>n) $, where $\varphi(x)$ is the first return time
from $(x,0)$ to $Y=\bar X\times\{0\}$, using the classical following renewal equation \cite{DE}:
$$\mathbf 1=\sum_{j=0}^{n} \mathbf 1_{\{\varphi>n-j\}}
   \circ \bar f^j.
      \mathbf 1_{\{S_j=0\}}\quad \mbox{on}\ \bar X\, $$
where $j$ plays the role of the last visit time to $Y$ before time $n$.
Hence, applying $P^n$, this leads to:
$$\mathbf 1=\sum_{j=0}^n U_{n-j}Q_{j,0},\quad\mbox{with}\ U_k:= P^k\left(\mathbf 1_{\{\varphi>k\}}\, \cdot\right).$$
This kind of properties has been used in \cite{FP09b} to study the asymptotic behaviour of the number of different obstacles visited by the Lorentz process up to time $n$, in \cite{ps10} to study some quantitative recurrence properties.
\end{enumerate}
\section{Example: Lorentz process}
Consider a $\mathbb Z^2$-periodic configuration of obstacles in the plane:
$O_i+\ell$, $i=1,...,I$, $\ell\in\mathbb Z^2$, with $I\ge 2$. We assume that the $O_i$ are convex open sets, with $C^3$-smooth boundary with non null curvature. We assume that the closures of any couple of distinct obstacles
$O_i+\ell$ and $O_j+m$ are disjoint.
The Lorentz process describes the displacement in $Q:=\mathbb R^2\setminus \bigcup_{\ell\in\mathbb Z^2}\bigcup_{i=1}^IO_{i,\ell}$ of a point particle 
moving with unit speed and with elastic reflection off the obstacles
(i.e. reflected angle=incident angle).
We assume that the horizon is finite, i.e. that each trajectory meets at least one obstacle.

We consider the dynamical system $(M,T,\nu)$ corresponding to the collision times, where $M$ is the set of reflected vectors, where $T:M\rightarrow M$ is the transformation mapping a reflected vector to the reflected vector at the next collision time and where $\nu$ is the invariant measure absolutely continuous with respect to the Lebesgue measure. 
For every $\ell\in\mathbb Z^2$, we write $\mathcal C_\ell$
for the set of reflected vectors which are based on $\bigcup_{i=1}^I(O_i+\ell)$. 
Up to a renormalization of $\nu$, we assume that
$\nu(\mathcal C_0)=1$. We call $\mathcal C_\ell$ the $\ell$-cell.

It is well known that $(M,T,\nu)$ can be represented as the
$\mathbb Z^2$-extension $(X,f,\mu)$ of $(\bar X,\bar f,\bar\mu)$
by $\psi:\bar X\rightarrow\mathbb Z^2$, where $\bar X=\mathcal C_0$, $\bar\mu=\bar\nu(\mathcal C_0\cap\cdot)$, where $\bar f$ and $\psi$
are such that $T(q,\vec v)=(q'+\psi(q,\vec v),\vec v')$
if $(q',\vec v')=\bar f(q,\vec v)$ ($\bar f$ corresponds to $T$ quotiented by the equality of positions modulo $\mathbb Z^2$).
 Note that $S_n(x):=\sum_{k=0}^{n-1}\psi\circ \bar f^k(x)$ is
the label of the cell in which the particle starting from configuration $x\in\mathcal C_0$ is at the $n$-th reflection time.

The dynamical system $(\bar X,\bar f,\bar\mu)$ is the Sinai
billiard \cite{Sinai70,ChernovMarkarian}. Central limit theorems in this context have been established
in \cite{BS81,BCS91,young}. 
In particular $(S_n/\sqrt{n})_n$
converges in distribution, with respect to $\bar\mu$ to 
a centered gaussian random variable $B$ with non-degenerate variance matrix $\Sigma$, so $\Phi_B(x)=e^{-\frac {\langle \Sigma x,x\rangle}2}/(2\pi\sqrt{\det \Sigma})$.

Let $R_0\subset \bar X$ be the set of reflected vectors that are tangent to $\bigcup_{i=1}^I\partial O_i$. The billiard map $\bar f$
defines a $C^1$-diffeomorphism from $\bar X\setminus(R_0\cup \bar f^{-1}R_0)$ onto  $\bar X\setminus(R_0\cup \bar f R_0)$.
For any integers $k\le k'$, we set $\xi_k^{k'}$ for the partition
of $\bar X\setminus \bigcup_{j=k}^{k'} \bar f^{-j}R_0$ in connected components and $\xi_k^\infty:=\bigvee _{j\ge k}\xi_k^j$.
For any $\bar u:\bar X\rightarrow \mathbb R$ and $-\infty< k\le k'\le \infty$, we define the following local continuity modulus:
$$\omega_k^{k'}(\bar u,\bar x):= \sup_{\bar y\in\xi_{k}^{k'}(\bar x)}|\bar u(\bar x)-\bar u(\bar y)|.$$

The following result is established thanks to the use of the towers constructed by Young in \cite{young}.
\begin{proposition}\label{pro:TLLbillZ2}
Let $p>1$. There exists $c>0$ such that, for any $k\ge 1$, for any measurable functions $\bar u,\bar v:\bar X\rightarrow \mathbb R$ such that $\bar u$ is $\xi_{-k}^k$-measurable and $\bar v$ is $\xi_{-k}^\infty$-measurable, for every $n>2k$ and for every $\ell\in\mathbb Z^2$,
$$\left|\mathbb E_{\bar\mu}\left[\bar u\, \mathbf 1_{\{S_n=\ell\}}\, \bar v\circ \bar f^k\right]-\frac{\Phi_B\left(\frac\ell{\sqrt{n-2k}}\right)}{n-2k}\int_{\bar X}\bar u\, d\bar\mu\, \int_{\bar  X}\bar v\, d\bar\mu\right|\le\frac{ck\Vert \bar v\Vert_p \Vert \bar u\Vert_\infty}{(n-2k)^{\frac 32}}.$$
\end{proposition}
\begin{proof}
The proof of this result is exactly the same as the proof of \cite[prop 4.1]{ps10}, by replacing $\mathbf 1_{A}$ by $\bar u$, $\mathbf 1_B$ by $\bar v$, $\mathbf 1_{\hat A}$ and $\mathbf 1_{\hat B}$ by respectively
$\hat u$ and $\hat v$ such that: $\hat u\circ\hat \pi
=\bar u\circ T^k\circ\tilde \pi$ and  $\hat v\circ\hat \pi
=\bar v\circ T^k\circ\tilde \pi$. With the notations of \cite{ps10},
we have $\sup_{t\in[-\pi,\pi]^2}\Vert P_t^kP^k\hat u\Vert \le c_0\Vert u\Vert_\infty$. So that $\bar\mu(B)^{1/p}$ of \cite[p. 865]{ps10} is replaced by $\Vert \bar u\Vert_\infty \Vert \bar v\Vert_p$.
\end{proof}

For any $u,v:M\rightarrow \mathbb R$ and $k\in\mathbb Z^2$, we set as previously: $u_k:=u(\cdot,k)$
and $v_k:=v(\cdot,k)$. 
\begin{theorem}
Let $p>1$ and $u,v:X\rightarrow \mathbb R$ measurable such that 
\begin{equation}\label{HYP1}
\sum_{\ell\in\mathbb Z^2}(\Vert u_\ell\Vert_\infty+\Vert v_\ell\Vert_p)<\infty \, ,
\end{equation}
\begin{equation}\label{HYP1bis}
\forall k\ge 1,\quad \sum_{\ell\in\mathbb Z^2}\Vert \omega_{-k}^{\infty} (v_\ell,\cdot)\Vert_p<\infty \, ,
\end{equation}
\begin{equation}\label{HYP2}
\lim_{k\rightarrow +\infty}\sum_{\ell\in\mathbb Z^2}
  \left(\Vert\omega_{-k}^k(u_\ell,\cdot)\Vert_1
  +\Vert \omega_{-k}^{\infty}(v_\ell,\cdot)\Vert_1\right)=0.
\end{equation}
Then
\begin{equation}\label{MIXBILL}
\int_X u.v\circ f^n\, d\mu= \frac{\Phi_B(0)}{n}\int_Xu\, d\mu
      \int_Xv\, d\mu+o(n^{-1}).
\end{equation}
\end{theorem}
\begin{proof}
It is enough to prove the result for non-negative $u,v$.
We assume from now on that $u,v$ take their values in $[0,+\infty)$.
Let $\ell\in\mathbb Z^2$ and let $k$ be a positive integer. We define $u_{\ell}^{(k,\pm)}$ and $v_{\ell}^{(k,\pm)}$:
$$u_{\ell}^{(k,-)}(\bar x):=\inf_{\bar y\in\xi_{-k}^k(\bar x)} u_\ell(\bar y), \quad u_{\ell}^{(k,+)}(\bar x):=\sup_{\bar y\in\xi_{-k}^k(\bar x)} u_\ell(\bar y) $$
and
$$v_{\ell}^{(k,-)}(\bar x):=\inf_{\bar y\in\xi_{-k}^\infty(x)} v_\ell(\bar y),\quad v_{\ell}^{(k,+)}(\bar x):=\sup_{\bar y\in\xi_{-k}^\infty(x)} v_\ell(\bar y).$$
Observe that
\begin{equation}\label{MA1}
u_\ell^{(k,+)}- u_\ell^{(k,-)}\le 2 \omega_{-k}^k(u_\ell,\cdot)
\end{equation}
and that
\begin{equation}\label{MA2}
v_\ell^{(k,+)}- v_\ell^{(k,-)}\le 2 \omega_{-k}^\infty(v_\ell,\cdot).
\end{equation}
We then consider $u^{(k,\pm)},v^{(k,\pm)}: X\rightarrow \mathbb R$ such that
$$\forall \ell\in\mathbb Z^2,\quad 
      u^{(k,\pm)}(\ell,\cdot)\equiv u_\ell^{(k,\pm)}\quad\mbox{and}\quad v^{(k,\pm)}(\ell,\cdot)\equiv v_\ell^{(k,\pm)}.$$
Note that
\begin{equation}\label{MA}
 u^{(k,-)}  \le u\le  u^{(k,+)}
\quad\mbox{and}\quad
v^{(k,-)}  \le v\le  v^{(k,+)}\,
\end{equation}
and so
\begin{equation}\label{COMPINT}
\int_X u^{(k,-)}.v^{(k,-)}\circ f^n \, d\mu \le \int_X u.v\circ f^n \, d\mu
      \le \int_X u^{(k,+)}.v^{(k,+)}\circ f^n \, d\mu\, .
\end{equation}
We have
\begin{eqnarray*}
  \int_X u^{(k,\pm)}.v^{(k,\pm)} \circ f^n \, d\mu&=& 
      \sum_{\ell,m\in\mathbb Z^2}
      \int_{\bar X} u_\ell^{(k,\pm)}\mathbf 1_{\{S_n=m-\ell\}}
         v_m^{(k,\pm)} \circ f^n \, d\mu.
\end{eqnarray*}
Applying Proposition \ref{pro:TLLbillZ2} to the couples
$(u_{\ell}^{(k,-)},v_m^{(k,-)})$ and $(u_{\ell}^{(k,+)},v_m^{(k,+)})$, for every $\ell,m\in \mathbb Z^2$, we obtain that
$$\left|\int_X u^{(k,\pm)}.v^{(k,\pm)} \circ f^n \, d\mu
   -\sum_{\ell,m}
   \frac{\Phi_B\left(\frac{m-\ell}{\sqrt{n-2k}}\right)}{n-2k}\int_{\bar X} u_\ell^{(k,\pm)}\, d\bar\mu\, \int_{\bar X} v_m^{(k,\pm)}\, d\bar\mu\right|
\le$$
$$
\le \sum_{\ell,m\in\mathbb Z^2}\frac{ck\Vert  v_m^{(k,\pm)}\Vert_p \Vert   u_\ell^{(k,\pm)}\Vert_\infty}{(n-2k)^{\frac 32}}=o(n^{-1}),
$$
due to \eqref{HYP1} and \eqref{HYP1bis}.
Hence
$$\int_X u^{(k,\pm)}.v^{(k,\pm)} \circ f^n \, d\mu
   =\frac 1{n-2k}\sum_{\ell,m}
   \Phi_B\left(\frac{m-\ell}{\sqrt{n-2k}}\right)\int_{\bar X} u_\ell^{(k,\pm)}\, d\bar\mu\, \int_{\bar X} v_m^{(k,\pm)}\, d\bar\mu +o(n^{-1}).$$
But $\Phi_B$ is continuous and bounded by $\Phi_B(0)$. 
Hence, due to the Lebesgue dominated convergence theorem,
we obtain
\begin{eqnarray}
\int_X u^{(k,\pm)}.v^{(k,\pm)} \circ f^n \, d\mu
   &=&\frac {\Phi_B(0)}{n-2k}\sum_{\ell,m}
   \int_{\bar X} u_\ell^{(k,\pm)}\, d\bar\mu\, \int_{\bar X} v_m^{(k,\pm)}\, d\bar\mu +o(n^{-1})\nonumber\\
   &=&\frac {\Phi_B(0)}{n}
   \int_{X} u^{(k,\pm)}\, d\mu\, \int_{X} v^{(k,\pm)}\, d\mu +o(n^{-1}).\label{MIXAPPROX}
\end{eqnarray}
Moreover \eqref{HYP2}, \eqref{MA1} and \eqref{MA2} imply that
$$\lim_{k\rightarrow +\infty}\int_X |u^{(k,\pm)}-u|\, d\mu 
     =\lim_{k\rightarrow +\infty}\int_X |v^{(k,\pm)}-v|\, d\mu 
     =0.$$
We conclude by combining this with
\eqref{COMPINT} and \eqref{MIXAPPROX}.
\end{proof}
As a consequence we obtain the mixing for dynamically Lipschitz
functions. Let $\vartheta\in(0,1)$. We set
$$d_\vartheta(x,y):=\vartheta^{s(x,y)}\, ,$$
where $s(x,y)$ is the maximum of the integers $k>0$ such that $x$ and $y$ lie in the same connected component of $M\setminus
\bigcup_{j=-k}^kT^{-j}S_0$, where $S_0$ is the set of vectors of $M$
tangent to $\partial Q$.
The function $s(\cdot,\cdot)$ is called {\bf separation time}.
We set
$$L_\vartheta(u):=\sup_{x\ne y}\frac{|u(x)-u(y)|}{d_\vartheta(x,y)} $$
for the Lipschitz constant of $u$ with respect to $d_\vartheta$.

It is worth noting that, for every $\eta\in(0,1]$, there exists
$\vartheta>0$ such that
every $\eta$-H\"older function (both in position-speed)
is dynamically Lipschitz continuous with respect to $\vartheta$.
\begin{corollary}
Assume that $u,v:M\rightarrow \mathbb R$ are bounded uniformly dynamically H\"older (in position and in speed) and that
\begin{equation}\label{HYP1A}
\sum_{\ell\in\mathbb Z^2}(\Vert u\mathbf 1_{\mathcal C_\ell}\Vert_\infty+\Vert v\mathbf 1_{\mathcal C_\ell}\Vert_\infty)<\infty \, ,
\end{equation}
and
\begin{equation}\label{HYP1bisA}
\sum_{\ell\in\mathbb Z^2}(L_\vartheta( u\mathbf 1_{\mathcal C_\ell})+L_\vartheta(v\mathbf 1_{\mathcal C_\ell}))<\infty \, .
\end{equation}
Then
\begin{equation}\label{MIXBILL1}
\int_X u.v\circ f^n\, d\mu= \frac{\Phi_B(0)}{n}\int_Xu\, d\mu
      \int_Xv\, d\mu+o(n^{-1}).
\end{equation}
\end{corollary}
\bigskip

{\bf Acknowledgment.\/} The author wishes to thank Marco Lenci
for having suggested this work.

%%%%%%%%%%%%%%%

\end{document}